\documentclass{article}

\usepackage{times}
\usepackage{graphicx} 
\usepackage{subfigure} 
\usepackage{amsmath, amssymb, math}
\usepackage{natbib}

\usepackage{algorithm}
\usepackage{algorithmic}

\usepackage{hyperref}

\newtheorem{lemm}{Proposition S\hspace{-2.5pt}}

\usepackage[accepted]{icml2017}

\icmltitlerunning{Dissipativity Theory for Nesterov's Accelerated Method}

\begin{document} 

\twocolumn[
\icmltitle{Dissipativity Theory for Nesterov's Accelerated Method}




\begin{icmlauthorlist}
\icmlauthor{Bin Hu}{goo}
\icmlauthor{Laurent Lessard}{goo}
\end{icmlauthorlist}
\icmlaffiliation{goo}{University of Wisconsin--Madison, Madison, WI 53706, USA}

\icmlcorrespondingauthor{Bin Hu}{bhu38@wisc.edu}

\icmlkeywords{Nesterov's accelerated method, dissipativity theory, Lyapunov theory}

\vskip 0.3in
]



\printAffiliationsAndNotice{}  

\begin{abstract} 
In this paper, we adapt the control theoretic concept of dissipativity theory to provide a natural understanding of Nesterov's accelerated method. Our theory ties rigorous convergence rate analysis to the physically intuitive notion of energy dissipation. Moreover, dissipativity allows one to efficiently construct Lyapunov functions (either numerically or analytically) by solving a small semidefinite program. Using novel supply rate functions, we show how to recover known rate bounds for Nesterov's method and we generalize the approach to certify both linear and sublinear rates in a variety of settings. Finally, we  link the continuous-time version of dissipativity to recent works on algorithm analysis that use discretizations of ordinary differential equations.
\end{abstract} 

\section{Introduction}
Nesterov's accelerated method~\citep{YEN03a} has garnered interest in the machine learning community because of its fast global convergence rate guarantees.
The original convergence rate proofs of Nesterov's accelerated method are derived using the method of estimate sequences,
which has proven difficult to interpret.
This observation motivated a sequence of recent works on new analysis and interpretations of Nesterov's accelerated method \citep{bubeck2015geometric, Lessard2014, Su2016, drusvyatskiy2016, flammarion2015, wibisono2016, wilson2016}.

Many of these recent papers rely on Lyapunov-based stability arguments. Lyapunov theory is an analogue to the principle of minimum energy and brings a physical intuition to convergence behaviors. When applying such proof techniques, one must construct a \textit{Lyapunov function}, which is a nonnegative function of the algorithm's state (an ``internal energy'') that decreases along all admissible trajectories. Once a Lyapunov function is found, one can relate the rate of decrease of this internal energy to the rate of convergence of the algorithm. The main challenge in applying Lyapunov's method is finding a suitable Lyapunov function.

There are two main approaches for Lyapunov function constructions.
The first approach adopts the integral quadratic constraint (IQC) framework~\citep{Megretski1997} from control theory and formulates a linear matrix equality (LMI) whose feasibility implies the linear convergence of the algorithm~\citep{Lessard2014}. Despite the generality of the IQC approach and the small size of the associated LMI, one must typically resort to numerical simulations to solve the LMI.
The second approach seeks an ordinary differential equation (ODE) that can be appropriately discretized to yield the algorithm of interest. One can then gain intuition about the trajectories of the algorithm by examining trajectories of the continuous-time ODE~\citep{Su2016, wibisono2016, wilson2016}. The work of \citet{wilson2016} also establishes a general equivalence between Lyapunov functions and estimate sequence proofs.

In this paper, we bridge the IQC approach \citep{Lessard2014} and the discretization approach \citep{wilson2016} by using dissipativity theory~\citep{willems72a, willems72b}. The term ``dissipativity'' is borrowed from the notion of energy dissipation in physics and the theory provides a general approach for the intuitive understanding and construction of Lyapunov functions. Dissipativity for quadratic Lyapunov functions in particular \citep{willems72b} has seen widespread use in controls. In the sequel, we tailor dissipativity theory to the automated construction of Lyapunov functions, which are not necessarily quadratic, for the analysis of optimization algorithms.
Our dissipation inequality leads to an LMI condition that is simpler than the one in \citet{Lessard2014} and hence more amenable to being solved analytically. When specialized to Nesterov's accelerated method, our LMI recovers the Lyapunov function proposed in \citet{wilson2016}.
Finally, we extend our LMI-based approach to the sublinear convergence analysis of Nesterov's accelerated method in both discrete and continuous time domains. This complements the original LMI-based approach in \citet{Lessard2014}, which mainly handles linear convergence rate analyses.

An LMI-based approach for sublinear rate analysis similar to ours was independently and simultaneously proposed by \citet{morari}.
While this work and the present work both draw connections to the continuous-time results mentioned above, different algorithms and function classes are emphasized. For example, \citet{morari} develops LMIs for gradient descent and proximal/projection-based variants with convex/quasi-convex objective functions. In contrast, the present work develops LMIs tailored to the analysis of discrete-time accelerated methods
 and Nesterov's method in particular.

\section{Preliminaries}

\subsection{Notation}
Let $\R$ and $\R_+$ denote the real and nonnegative real numbers, respectively.
 Let $I_p$ and $0_p$ denote the $p\times p$ identity and zero matrices, respectively.
The Kronecker product of two matrices is denoted $A \otimes B$
and satisfies the properties $(A\otimes B)^\tp=A^\tp \otimes B^\tp$ and $(A\otimes B)(C\otimes D)=(AC)\otimes (BD)$ when the matrices have compatible dimensions. Matrix inequalities hold in the semidefinite sense unless otherwise indicated.
%
%
A differentiable function $f:\R^p\to \R$ is $m$-strongly convex if  
$f(x)\ge f(y)+\nabla f(y)^\tp (x-y)+\frac{m}{2} \|x-y\|^2$ for all $x,y \in \R^p$ and is $L$-smooth if  $\|\nabla f(x)-\nabla f(y)\|\le L \|x-y\|$ for all $x,y\in \R^p$. Note that $f$ is convex if $f$ is $0$-strongly convex.
We use $x_\star$ to denote a point satisfying $\nabla f(x_\star)=0$. When $f$ is $L$-smooth and $m$-strongly convex, $x_\star$ is unique.

\subsection{Classical Dissipativity Theory}

Consider a linear dynamical system governed by the state-space model
\begin{align}
\label{eq:sys1}
\xi_{k+1}=A\xi_k+B w_k.
\end{align}
Here, $\xi_k\in \R^{n_\xi}$ is the state, $w_k\in \R^{n_w}$ is the input, $A\in \R^{{n_\xi}\times n_\xi}$ is the state transition matrix, and $B\in \R^{{n_\xi}\times n_w}$ is the input matrix. The input $w_k$ can be physically interpreted as a \textit{driving force}.
Classical dissipativity theory describes how the internal energy stored in the state $\xi_k$ evolves with time $k$ as one applies the input $w_k$ to drive the system. A key concept in dissipativity theory is the \textit{supply rate}, which characterizes the energy change in $\xi_k$ due to the driving force $w_k$. The supply rate is a function $S:\R^{n_\xi}\times \R^{n_w}\to \R$ that maps any state/input pair $(\xi, w)$ to a scalar measuring the amount of energy delivered from $w$ to state $\xi$. Now we introduce the notion 
of dissipativity.

\begin{defn}
The dynamical system \eqref{eq:sys1} is dissipative with respect to the supply rate $S$ if there exists a function $V:\R^{n_\xi}\to \R_+$ such that $V(\xi) \ge 0$ for all $\xi \in \R^{n_\xi}$ and
\begin{align}
\label{eq:DI}
V(\xi_{k+1})- V(\xi_k)\le S(\xi_k, w_k)
\end{align}
for all $k$. The function $V$ is called a storage function, which quantifies the energy stored in the state $\xi$. In addition, \eqref{eq:DI} is called the dissipation inequality.
\end{defn}

The dissipation inequality \eqref{eq:DI} states that the change of the internal energy stored in $\xi_k$ is equal to the difference between the supplied energy and the dissipated energy. Since there will always be some energy dissipating from the system, the change in the stored energy (which is exactly $V(\xi_{k+1})-V(\xi_k)$) is always bounded above by the energy supplied to the system (which is exactly $S(\xi_k, w_k)$).
A variant of \eqref{eq:DI} known as the exponential dissipation inequality states that for some $0\le \rho < 1$, we have
\begin{equation}\label{eq:DI1}
V(\xi_{k+1})- \rho^2 V(\xi_k)\le S(\xi_k, w_k),
\end{equation}
which states that at least a fraction $(1-\rho^2)$ of the internal energy will dissipate at every step.

The dissipation inequality \eqref{eq:DI1} provides a direct way to construct a Lyapunov function based on the storage function.
It is often the case that we have prior knowledge about how the driving force $w_k$ is related to the state $\xi_k$. Thus, we may know additional information about the supply rate function $S(\xi_k,w_k)$. For example, if $S(\xi_k, w_k)\le 0$ for all $k$ then \eqref{eq:DI1} directly implies that $V(\xi_{k+1})\le \rho^2 V(\xi_k)$, and the storage function $V$ can serve as a Lyapunov function. The condition $S(\xi_k, w_k)\le 0$ means
that the driving force $w_k$ does not inject any energy into the system and may even extract energy out of the system. Then, the internal energy will decrease no slower than the linear rate $\rho^2$ and approach a minimum value at equilibrium.

An advantage of dissipativity theory is that for any quadratic supply rate,  one can automatically construct the dissipation inequality using semidefinite programming.
{We now state a standard result from the controls literature.

\begin{thm}
\label{thm:DI}
Consider the following quadratic supply rate with $X\in \R^{(n_{\xi}+n_w)\times (n_{\xi}+n_w)}$ and $X=X^\tp$.
\begin{align}
\label{eq:supply}
S(\xi, w) \defeq \bmat{\xi \\ w}^\tp X \bmat{\xi \\ w}.
\end{align}
If there exists a matrix $P\in \R^{n_{\xi}\times n_{\xi}}$ with $P\ge 0$ such that
\begin{align}
\label{eq:lmi1}
\bmat{A^\tp P A-\rho^2 P & A^\tp P B \\ B^\tp P A & B^\tp P B}-X\le 0,
\end{align}
then the dissipation inequality~\eqref{eq:DI1} holds for all trajectories of~\eqref{eq:sys1} with $V(\xi)\defeq\xi^\tp P \xi$.
\end{thm}
\begin{proof}
Based on the state-space model \eqref{eq:sys1}, we have
\begin{align*}
\begin{split}
V(\xi_{k+1})&=\xi_{k+1}^\tp P \xi_{k+1}\\
&=(A\xi_k+Bw_k)^\tp P (A\xi_k+Bw_k)\\
&=\bmat{\xi_k \\ w_k}^\tp \bmat{A^\tp P A & A^\tp P B \\ B^\tp P A & B^\tp P B} \bmat{\xi_k \\ w_k}.
\end{split}
\end{align*}
Hence we can left and
  right multiply \eqref{eq:lmi1}  by $\bmat{\xi_k^\tp & w_k^\tp}$ and $\bmat{\xi_k^\tp & w_k^\tp}^\tp$, and directly obtain the desired conclusion.
\end{proof}
The left-hand side of~\eqref{eq:lmi1} is linear in $P$, so~\eqref{eq:lmi1} is a \emph{linear matrix inequality} (LMI) for any fixed $A,B,X,\rho$. The set of $P$ such that \eqref{eq:lmi1} holds is therefore a convex set and can be efficiently searched using interior point methods, for example.
To apply the dissipativity theory for linear convergence rate analysis, one typically follows two steps.
\begin{enumerate}
\item Choose a proper quadratic supply rate function $S$ satisfying certain desired properties, e.g. $S(\xi_k, w_k)\le 0$.
\item Solve the LMI \eqref{eq:lmi1} to obtain a storage function $V$, which is then used to construct a Lyapunov function.
\end{enumerate}

In step 2, the LMI obtained is typically very small, e.g. $2\times 2$ or $3\times 3$, so we can often solve the LMI analytically.
For illustrative purposes, we rephrase the existing LMI analysis of the gradient descent method \citep[\S4.4]{Lessard2014} using the notion of dissipativity.

\subsection{Example: Dissipativity for Gradient Descent}\label{sec:grad1}
There is an intrinsic connection between dissipativity theory and the IQC approach \citep{megretskiCSH, seiler13}. The IQC analysis of the gradient descent method in \citet{Lessard2014} may be reframed using dissipativity theory. Then, the pointwise IQC \citep[Lemma 6]{Lessard2014} amounts to using a quadratic supply rate $S$ with $S\le 0$. Specifically, assume $f$ is $L$-smooth and $m$-strongly convex,
and consider the gradient descent method
\begin{align}\label{eq:gradient_descent}
x_{k+1}=x_k-\alpha \nabla f(x_k).
\end{align}
We have $x_{k+1}-x_\star=x_k-x_\star -\alpha \nabla f(x_k)$, where $x_\star$ is the unique point satisfying $\nabla f(x_\star)=0$. Define $\xi_k\defeq x_k-x_\star$ and $w_k \defeq \nabla f(x_k)$. Then the gradient descent method is modeled by \eqref{eq:sys1} with $A\defeq I_p$ and $B\defeq -\alpha I_p$. Since $w_k=\nabla f(\xi_k+x_\star)$, we can  define the following quadratic supply rate
\begin{align}
\label{eq:supply1}
S(\xi_k, w_k)=\bmat{\xi_k \\ w_k}^\tp \bmat{2m LI_p & -(m+L) I_p \\ -(m+L)I_p & 2I_p}\bmat{ \xi_k \\ w_k}
\end{align}
By co-coercivity, we have $S(\xi_k, w_k)\le 0$ for all $k$. This just restates \citet[Lemma 6]{Lessard2014}.
Then, we can directly apply Theorem \ref{thm:DI} to construct the dissipation inequality. We can parameterize $P=p\otimes I_p$  and define the storage function as $V(\xi_k)=p\norm{\xi_k}^2=p\norm{x_k-x_\star}^2$. The LMI \eqref{eq:lmi1} becomes
\begin{align*}
\left(\bmat{(1-\rho^2)p & -\alpha p \\ -\alpha p & \alpha^2 p}+\bmat{-2m L& m+L \\ m+L & -2}\right)\otimes I_p \le 0.
\end{align*}
Hence for any $0\le \rho <1$, we have $p\norm{x_{k+1}-x_\star}^2 \le \rho^2 p\norm{x_k-x_\star}^2$ if there exists $p\ge 0$ such that
\begin{align}
\label{eq:lmiGD}
\bmat{(1-\rho^2)p & -\alpha p \\ -\alpha p & \alpha^2 p}+\bmat{-2m L& m+L \\ m+L & -2}\le 0
\end{align}
The LMI~\eqref{eq:lmiGD} is simple and can be analytically solved to recover the existing rate results for the gradient descent method. For example, we can choose $(\alpha,\rho,p)$ to be $(\frac{1}{L},1-\frac{m}{L},L^2)$ or $(\frac{2}{L+m},\frac{L-m}{L+m},\frac{1}{2}(L+m)^2$) to immediately recover the standard rate results in~\citet{polyak}.

Based on the example above, it is evident that choosing a proper supply rate is critical for the construction of a Lyapunov function. The supply rate \eqref{eq:supply1} turns out to be inadequate for the analysis of Nesterov's accelerated method. For Nesterov's accelerated method, the dependence between the internal energy and the driving force is more complicated due to the presence of momentum terms. We will next develop a new supply rate that captures this complicated dependence. 
We will also make use of this new supply rate to recover the standard linear rate results for Nesterov's accelerated method.

\section{Dissipativity for Accelerated Linear Rates}
\subsection{Dissipativity for Nesterov's Method}
Suppose $f$ is $L$-smooth and $m$-strongly convex with $m>0$. Let $x_\star$ be the unique point satisfying $\nabla f(x_\star)=0$.
Now we consider Nesterov's accelerated method, which uses the following iteration rule to find $x_\star$:
\begin{subequations}\label{eq:NAG}
\begin{align}
x_{k+1}&=y_k-\alpha \nabla f(y_k), \\
y_k&=(1+\beta) x_k-\beta x_{k-1}.
\end{align}
\end{subequations}
We can rewrite \eqref{eq:NAG} as
\begin{align}
\label{eq:fdNAG}
\begin{split}
\bmat{x_{k+1}-x_\star \\ x_k-x_\star}=A\bmat{x_k-x_\star \\ x_{k-1}-x_\star}+Bw_k
\end{split}
\end{align}
where $w_k\defeq\nabla f(y_k)=\nabla f\left((1+\beta)x_k-\beta x_{k-1}\right)$. Also, $A\defeq\tilde{A}\otimes I_p$, $B\defeq\tilde{B}\otimes I_p$, and $\tilde{A}, \tilde{B}$ are defined by
\begin{align}
\label{eq:NAGAB}
\tilde{A}\defeq\bmat{1+\beta & -\beta \\ 1 & 0}, \,\quad \tilde{B}\defeq\bmat{-\alpha \\ 0}.
\end{align}
Hence, Nesterov's accelerated method \eqref{eq:NAG} is in the form of \eqref{eq:sys1} with $\xi_k=\bmat{(x_k-x_\star)^\tp & (x_{k-1}-x_\star)^\tp}^\tp$.

Nesterov's accelerated method can improve the convergence rate since the input $w_k$ depends on both $x_k$ and $x_{k-1}$, and drives the state in a specific direction, i.e. along $(1+\beta)x_k-\beta x_{k-1}$.
This leads to a supply rate that extracts energy out of the system significantly faster than with gradient descent. This is formally stated in the next lemma.


\begin{lem}
\label{lem:SupplyNAG}
Let $f$ be $L$-smooth and $m$-strongly convex with $m>0$. Let  $x_\star$ be the unique point satisfying $\nabla f(x_\star)=~0$.
Consider Nesterov's method~\eqref{eq:NAG} or equivalently \eqref{eq:fdNAG}. The following inequalities hold for all trajectories.
\begin{align*}
\bmat{x_k-x_\star\\x_{k-1}-x_\star \\ \nabla f(y_k)}^\tp X_1 \bmat{x_k-x_\star\\x_{k-1}-x_\star \\ \nabla f(y_k) } &\le f(x_k)-f(x_{k+1})\\
\bmat{x_k-x_\star\\x_{k-1}-x_\star \\ \nabla f(y_k)}^\tp X_2  \bmat{x_k-x_\star\\x_{k-1}-x_\star \\ \nabla f(y_k) } &\le f(x_\star)-f(x_{k+1})
\end{align*}
where $X_i=\tilde{X}_i\otimes I_p$ for $i=1,2$, and $\tilde{X}_i$ are defined by
\begin{align}
\label{eq:NAGX1}
\tilde{X}_1&\defeq\frac{1}{2}\!\bmat{ \beta^2 m & -\beta^2 m & -\beta\\ -\beta^2 m  & \beta^2 m & \beta\\ -\beta & \beta & \alpha(2-L\alpha)} \\
\label{eq:NAGX2}
\tilde{X}_2&\defeq\frac{1}{2}\!\addtolength{\arraycolsep}{-3pt}\bmat{ (1+\beta)^2 m & -\beta(1+\beta) m & -(1+\beta)\\ -\beta(1+\beta) m  & \beta^2 m & \beta\\ -(1+\beta) & \beta & \alpha(2-L\alpha)}\!\!
\end{align}
Given any $0\le \rho\le 1$, one can define the supply rate as \eqref{eq:supply} with a particular choice of $X\defeq\rho^2 X_1+(1-\rho^2)X_2$. Then this supply rate satisfies the condition
\begin{multline}
\label{eq:NAGSupply}
S(\xi_k, w_k) \le \rho^2 (f(x_k)-f(x_\star))\\-(f(x_{k+1})-f(x_\star)).
\end{multline}
\end{lem}
\begin{proof}
The proof is similar to the proof of (3.23)--(3.24) in \citet{bubeck2015}, but \citet[Lemma 3.6]{bubeck2015} must be modified to account for the strong convexity of $f$. See the supplementary material for a detailed proof.
\end{proof}
The supply rate~\eqref{eq:NAGSupply} captures how the driving force $w_k$ is impacting the future state $x_{k+1}$. The physical interpretation is that there is some amount of  hidden energy in the system that takes the form of $f(x_k)-f(x_\star)$.  The supply rate condition \eqref{eq:NAGSupply} describes how the driving force $w_k$ is coupled with the hidden energy in the future. It says the delivered energy is bounded  by a weighted decrease of the hidden energy. Based on this supply rate, one can search Lyapunov function using the following theorem.

\begin{thm}
\label{thm:theNAG}
Let $f$ be $L$-smooth and $m$-strongly convex with $m>0$. Let  $x_\star$ be the unique point satisfying $\nabla f(x_\star)=~0$.
Consider Nesterov's accelerated method~\eqref{eq:NAG}.
For any rate $0\le \rho<1$, set $\tilde{X}\defeq\rho^2 \tilde{X}_1 +(1-\rho^2)\tilde{X}_2$ where $\tilde{X}_1$ and $\tilde{X}_2$ are defined in \eqref{eq:NAGX1}--\eqref{eq:NAGX2}. In addition, let $\tilde{A}, \tilde{B}$ be defined by \eqref{eq:NAGAB}.
If there exists a matrix $\tilde P \in \R^{2\times 2}$ with $\tilde P \ge 0$ such that
\begin{align}
\label{eq:lmi2}
\bmat{\tilde{A}^\tp \tilde{P} \tilde{A}-\rho^2 \tilde{P} & \tilde{A}^\tp \tilde{P} \tilde{B}\\ \tilde{B}^\tp \tilde{P} \tilde{A} & \tilde{B}^\tp \tilde{P} \tilde{B}}-\tilde{X}\le 0
\end{align}
then set $P\defeq\tilde{P}\otimes I_p$ and define the Lyapunov function
\begin{align}
\label{eq:LyaFun}
\mathcal{V}_k \defeq\bmat{x_k-x_\star\\x_{k-1}-x_\star }^\tp P \bmat{x_k-x_\star\\x_{k-1}-x_\star }+f(x_k)-f(x_\star),
\end{align}
which satisfies $\mathcal{V}_{k+1}\le \rho^2 \mathcal{V}_k$ for all $k$. Moreover, we have $f(x_k)-f(x_\star)\le \rho^{2k} \mathcal{V}_0\,$ for Nesterov's method.
\end{thm}
\begin{proof}
Take the Kronecker product of \eqref{eq:lmi2} and $I_p$, and hence \eqref{eq:lmi1} holds with $A\defeq\tilde{A}\otimes I_p$, $B\defeq\tilde{B}\otimes I_p$, and $X\defeq\tilde{X}\otimes I_p$. Let the supply rate $S$ be defined by \eqref{eq:supply}.
Then, define the quadratic storage function $V(\xi_k)\defeq\xi_k^\tp P \xi_k$ and apply Theorem \ref{thm:DI} to show $V(\xi_{k+1})-\rho^2 V(\xi_k)\le S(\xi_k, w_k)$. Based on the supply rate condition \eqref{eq:NAGSupply}, we can define the Lyapunov function $\mathcal{V}_k\defeq V(\xi_k)+f(x_k)-f(x_\star)$ and show $\mathcal{V}_{k+1}\le \rho^2\mathcal{V}_k$. Finally, since $P\ge 0$, we have $f(x_k)-f(x_\star)\le \rho^{2k} \mathcal{V}_0$.
\end{proof}

We can immediately recover the proposed Lyapunov function in \citet[Theorem 6]{wilson2016} by setting $\tilde{P}$ to
\begin{align}
\label{eq:PNAG}
\tilde{P}=\bmat{\sqrt{\frac{L}{2}}\\[2mm] \sqrt{\frac{m\vphantom{L}}{2}}-\sqrt{\frac{L}{2}}}\bmat{\sqrt{\frac{L}{2}}& \sqrt{\frac{m\vphantom{L}}{2}}-\sqrt{\frac{L}{2}}}.
\end{align}
Clearly $\tilde{P} \ge 0$. Now define $\kappa\defeq \frac{L}{m}$. Given $\alpha=\frac{1}{L}$, $\beta=\frac{\sqrt{\kappa}-1}{\sqrt{\kappa}+1}$, and $\rho^2=1-\sqrt{\frac{m}{L}}$, it is straightforward to verify that the left side of the LMI \eqref{eq:lmi1} is equal to 
\begin{align*}
\frac{m(\sqrt{\kappa}-1)^3}{2(\kappa+\sqrt{\kappa})}\bmat{-1 & 1 & 0\\1 & -1 & 0\\ 0 & 0 & 0},
\end{align*}
which is clearly negative semidefinite. Hence we can immediately construct a Lyapunov function using \eqref{eq:LyaFun} to prove the linear rate $\rho^2=1-\sqrt{\frac{m}{L}}$.

Searching for analytic certificates such as~\eqref{eq:PNAG} can either be carried out by directly analyzing the LMI, or by using numerical solutions to guide the search. For example, numerically solving~\eqref{eq:lmi2} for any fixed $L$ and $m$ directly yields \eqref{eq:PNAG}, which makes finding the analytical expression easy.

\subsection{Dissipativity Theory for More General Methods}
We demonstrate the generality of the dissipativity theory on a more general variant of Nesterov's method.
Consider a modified accelerated method 
\begin{subequations}\label{eq:GAM}
\begin{align}
x_{k+1}&=(1+\beta) x_k-\beta x_{k-1}-\alpha \nabla f(y_k), \\
y_k&=(1+\eta) x_k-\eta x_{k-1}.
\end{align}
\end{subequations}
When $\beta=\eta$, we recover Nesterov's accelerated method. When $\eta=0$, we recover the Heavy-ball method of~\citet{polyak}.
We can rewrite \eqref{eq:GAM} in state-space form \eqref{eq:fdNAG}
where $w_k\defeq\nabla f(y_k)=\nabla f\left((1+\eta)x_k-\eta x_{k-1}\right)$, $A\defeq\tilde{A}\otimes I_p$, $B\defeq\tilde{B}\otimes I_p$, and $\tilde{A}, \tilde{B}$ are defined by
\begin{align*}
\tilde{A}\defeq\bmat{1+\beta & -\beta \\ 1 & 0}, \,\quad \tilde{B}\defeq\bmat{-\alpha \\ 0}.
\end{align*}

\begin{lem}
\label{lem:SupplyGAM}
Let $f$ be $L$-smooth and $m$-strongly convex with $m>0$. Let  $x_\star$ be the unique point satisfying $\nabla f(x_\star)=~0$.
Consider the general accelerated method \eqref{eq:GAM}. Define the state  $\xi_k\defeq\bmat{(x_k-x_\star)^\tp & (x_{k-1}-x_\star)^\tp}^\tp$ and the input $w_k\defeq\nabla f(y_k)=\nabla f((1+\eta)x_k-\eta x_{k-1})$. Then the following inequalities hold for all trajectories.
\begin{align}\label{eq:GAMSupply1}
\bmat{\xi_k \\ w_k}^\tp (X_1+X_2) \bmat{\xi_k\\ w_k} &\le f(x_k)-f(x_{k+1})\\
\label{eq:GAMSupply2}
\bmat{\xi_k \\ w_k}^\tp  (X_1+X_3) \bmat{\xi_k \\ w_k} &\le f(x_\star)-f(x_{k+1})
\end{align} with $X_i\defeq\tilde{X}_i\otimes I_p\,$ for $i=1,2, 3$, and $\tilde{X}_i$ are defined by
\begin{align*}
\tilde{X}_1&\defeq\frac{1}{2}\bmat{ -L\delta^2 & L\delta^2  & -(1-L\alpha)\delta\\ L\delta^2  & -L\delta^2 & (1-L\alpha)\delta\\ -(1-L\alpha)\delta & (1-L\alpha)\delta & \alpha(2-L\alpha)} \\
\tilde{X}_2&\defeq\frac{1}{2}\bmat{ \eta^2 m & -\eta^2 m & -\eta\\ -\eta^2 m  & \eta^2 m & \eta\\ -\eta & \eta & 0} \\
\tilde{X}_3&\defeq\frac{1}{2}\bmat{ (1+\eta)^2 m & -\eta(1+\eta) m & -(1+\eta)\\ -\eta(1+\eta) m  & \eta^2 m & \eta\\ -(1+\eta) & \eta & 0}
\end{align*}
with $\delta\defeq \beta-\eta$. 
In addition, one can define the supply rate as \eqref{eq:supply} with $X\defeq X_1+\rho^2 X_2+(1-\rho^2)X_3$. Then for all trajectories $(\xi_k, w_k)$ of the general accelerated method \eqref{eq:GAM}, this supply rate satisfies the inequality
\begin{multline}\label{supp}
S(\xi_k, w_k)\le \rho^2 (f(x_{k})-f(x_\star))\\
-(f(x_{k+1})-f(x_\star)).
\end{multline}
\end{lem}
\begin{proof}
A detailed proof is presented in the supplementary material. One mainly needs to modify the proof by taking the difference between $\beta$ and $\eta$ into accounts.
\end{proof}

Based the supply rate~\eqref{supp}, we can immediately modify Theorem \ref{thm:theNAG} to handle the more general algorithm \eqref{eq:GAM}. 
Although we do not have general analytical formulas for the convergence rate of \eqref{eq:GAM}, preliminary numerical results suggest that there are a family of $(\alpha, \beta, \eta)$ leading to the rate $\rho^2=1-\sqrt{\frac{m}{L}}$, and the required value of $\tilde{P}$ is quite different from  \eqref{eq:PNAG}. This indicates that our proposed LMI approach could go beyond the Lyapunov function~\eqref{eq:PNAG}.

\begin{rem}
It is noted in \citet[\S3.2]{Lessard2014} that searching over combinations of multiple IQCs may yield improved rate bounds. The same is true of supply rates. For example, we could include $\lambda_1,\lambda_2 \ge 0$ as decision variables and search for a dissipation inequality with supply rate $\lambda_1 S_1 + \lambda_2 S_2$ where e.g. $S_1$ is~\eqref{eq:supply1} and $S_2$ is~\eqref{supp}.
\end{rem}

\section{Dissipativity for Sublinear Rates}

The LMI approach in \cite{Lessard2014} is tailored for the analysis of linear convergence rates for algorithms that are time-invariant (the $A$ and $B$ matrices in~\eqref{eq:sys1} do not change with $k$).
We now show that dissipativity theory can be used to analyze the sublinear rates $O(1/k)$ and $O(1/k^2)$ via slight modifications of the dissipation inequality.

\subsection{Dissipativity for $O(1/k)$ rates}

The $O(1/k)$ modification, which we present first, is very similar to the linear rate result.
 
\begin{thm}\label{thm6}
Suppose $f$ has a finite minimum $f_\star$.
Consider the LTI system \eqref{eq:sys1} with a supply rate satisfying
\begin{align}
\label{eq:supplySub1}
S(\xi_k, w_k)\le -(f(z_k)-f_\star)
\end{align}
for some sequence $\{z_k\}$.
If there exists a nonnegative storage function $V$ such that the dissipation inequality~\eqref{eq:DI} holds over  all trajectories of $(\xi_k, w_k)$, then the following inequality holds over all trajectories as well.
\begin{align}
\label{eq:sub1}
\sum_{k=0}^T (f(z_k)-f_\star)\le V(\xi_0).
\end{align}
In addition, we have the sublinear convergence rate
\begin{align}
\label{eq:subR1}
\min_{k:k\le T} (f(z_k)-f_\star)\le \frac{V(\xi_0)}{T+1}.
\end{align}
If $f(z_{k+1})\le f(z_k)$ for all $k$, then \eqref{eq:subR1} implies that\\ $f(z_k)-f_\star\le \frac{V(\xi_0)}{k+1}$ for all $k$.
\end{thm}
\begin{proof}
By the supply rate condition \eqref{eq:supplySub1} and the dissipation inequality \eqref{eq:DI}, we immediately get
\begin{align*}
V(\xi_{k+1})-V(\xi_k)+f(z_k)-f_\star\le 0.
\end{align*}
Summing the above inequality from $k=0$ to $T$ and using $V\ge 0$ yields the desired result.
\end{proof}
To address the sublinear rate analysis, the critical step is to choose an appropriate supply rate. If $f$ is $L$-smooth and convex, this is easily done. Consider the gradient method~\eqref{eq:gradient_descent} and define the quantities $\xi_k\defeq x_k-x_\star$, $A\defeq I_p$, and $B\defeq -\alpha I_p$ as in Section~\ref{sec:grad1}. Since  $f$ is $L$-smooth and convex, define the quadratic supply rate
\begin{align*}
S(\xi_k, w_k)\defeq\bmat{\xi_k \\ w_k}^\tp \bmat{0_p & -\frac{1}{2}I_p \\ -\frac{1}{2}I_p & \frac{1}{2L}I_p}\bmat{ \xi_k \\ w_k},
\end{align*}
which satisfies $S(\xi_k, w_k)\le f_\star-f(x_k)$ for all $k$ (co-coercivity).
Then we can directly apply the LMI \eqref{eq:lmi1} with $\rho=1$ to construct the dissipation inequality. Setting $P=p\otimes I_p$  and defining the storage function as $V(\xi_k)\defeq p\norm{\xi_k}^2=p\norm{x_k-x_\star}^2$, the LMI \eqref{eq:lmi1} becomes
\begin{align*}
\left(\bmat{0 & -\alpha p \\ -\alpha p & \alpha^2 p}+\bmat{0 & \frac{1}{2} \\ \frac{1}{2} & -\frac{1}{2L}}\right)\otimes I_p \le 0,
\end{align*}
which is equivalent to
\begin{align}
\label{eq:GD1}
\bmat{0 & -\alpha p+\frac{1}{2} \\ -\alpha p+\frac{1}{2} & \alpha^2 p-\frac{1}{2L}} \le 0.
\end{align}
Due to the $(1,1)$ entry being zero, \eqref{eq:GD1} holds if and only if
\[
\left\{\begin{aligned}
-\alpha p + \tfrac{1}{2}  &= 0 \\
\alpha^2 p - \tfrac{1}{2L} &\le 0
\end{aligned}\right.
\quad\implies\quad
\left\{\begin{aligned}
p &= \tfrac{1}{2\alpha} \\
\alpha &\le \tfrac{1}{L}
\end{aligned}\right.
\]
We can choose $\alpha=\frac{1}{L}$ and 
the bound~\eqref{eq:subR1} becomes
\begin{equation*}
\min_{k\le T}(f(x_k)-f_\star) \le \frac{L\norm{x_0-x_\star}^2}{2(T+1)}.
\end{equation*}
Since gradient descent has monotonically nonincreasing iterates, that is $f(x_{k+1})\le f(x_k)$ for all $k$,  we immediately recover the standard $O(1/k)$ rate result.

\subsection{Dissipativity for $O(1/k^2)$ rates}

Certifying a $O(1/k)$ rate for the gradient method required solving a single LMI~\eqref{eq:GD1}. However, this is not the case for the $O(1/k^2)$ rate analysis of Nesterov's accelerated method. Nesterov's algorithm has parameters that depend on $k$ so the analysis is more involved. We will begin with the general case and then specialize to Nesterov's algorithm. Consider the dynamical system
\begin{align}
\label{eq:sys2}
\xi_{k+1}=A_k \xi_k+B_k w_k
\end{align}
The state matrix $A_k$ and input matrix $B_k$ change with the time step $k$, and hence \eqref{eq:sys2} is referred to as a ``linear time-varying" (LTV) system.
The analysis of LTV systems typically requires a time-dependent supply rate such as
\begin{align}
\label{eq:supplysub}
S_k(\xi_k, w_k)\defeq\bmat{\xi_k \\ w_k}^\tp X_k \bmat{\xi_k \\ w_k}
\end{align}
If there exists a sequence $\{P_k\}$ with $P_k \ge 0$ such that
\begin{align}
\label{eq:lmiLTV}
\bmat{A_k^\tp P_{k+1} A_k- P_k & A_k^\tp P_{k+1} B_k\\[1mm]
B_k^\tp P_{k+1} A_k & B_k^\tp P_{k+1} B_k}- X_k\le 0
\end{align}
for all $k$,
then we have $V_{k+1}(\xi_{k+1})- V_k(\xi_k)\le S_k(\xi_k, w_k)$ with the time-dependent storage function defined as $V_k(\xi_k)\defeq\xi_k^\tp P_k \xi_k$.
This is a standard approach for dissipation inequality constructions of  LTV systems and can be proved using the same proof technique in Theorem \ref{thm:DI}.
Note that we need \eqref{eq:lmiLTV} to simultaneously hold for all $k$. This leads to an infinite number of LMIs in general.

Now we consider Nesterov's accelerated method for a convex $L$-smooth objective function $f$ \citep{YEN03a}.
\begin{subequations}
\label{eq:NAGsub}
\begin{align}
x_{k+1}&=y_k-\alpha_k \nabla f(y_k), \\
y_k&=(1+\beta_k) x_k-\beta_k x_{k-1}.
\end{align}
\end{subequations}
It is known that \eqref{eq:NAGsub} achieves a rate of $O(1/k^2)$ when $\alpha_k\defeq 1/L$ and $\beta_k$ is defined recursively as follows.
\[
\zeta_{-1}=0, \quad \zeta_{k+1}=\frac{1+\sqrt{1+4\zeta_k^2}}{2}, \quad \beta_k=\frac{\zeta_{k-1}-1}{\zeta_{k}}.
\]
The sequence $\{\zeta_k\}$ satisfies $\zeta_{k}^2-\zeta_k=\zeta_{k-1}^2$.
We now present a dissipativity theory for the sublinear rate analysis of Nesterov's accelerated method. 
Rewrite \eqref{eq:NAGsub} as 
\begin{align}
\label{eq:fdNAGsub}
\begin{split}
\bmat{x_{k+1}-x_\star \\ x_k-x_\star}=A_k\bmat{x_k-x_\star \\ x_{k-1}-x_\star}+B_kw_k
\end{split}
\end{align}
where $w_k\defeq\nabla f(y_k)=\nabla f\left((1+\beta_k)x_k-\beta_k x_{k-1}\right)$, $A_k\defeq\tilde{A}_k\otimes I_p$, $B_k\defeq\tilde{B}_k\otimes I_p$, and $\tilde{A}_k, \tilde{B}_k$ are given~by
\begin{align*}
\tilde{A}_k\defeq\bmat{1+\beta_k & -\beta_k \\ 1 & 0}, \,\quad \tilde{B}_k\defeq\bmat{-\alpha_k \\ 0}.
\end{align*}
Hence, Nesterov's accelerated method \eqref{eq:NAGsub} is in the form of \eqref{eq:sys2} with $\xi_k\defeq\bmat{(x_k-x_\star)^\tp & (x_{k-1}-x_\star)^\tp}^\tp$. The $O(1/k^2)$ rate analysis of Nesterov's method \eqref{eq:NAGsub} requires the following time-dependent supply rate.

\begin{lem}
\label{lem:SupplyNAGsub}
Let $f$ be $L$-smooth and convex. Let  $x_\star$ be a point satisfying $\nabla f(x_\star)=~0$. In addition, set $f_\star:=f(x_\star)$.
Consider Nesterov's method \eqref{eq:NAGsub} or equivalently \eqref{eq:fdNAGsub}. The following inequalities hold for all trajectories and for all $k$.
\begin{align*}
\bmat{x_k-x_\star\\x_{k-1}-x_\star \\ \nabla f(y_k)}^\tp M_ k \bmat{x_k-x_\star\\x_{k-1}-x_\star \\ \nabla f(y_k) } &\le f(x_k)-f(x_{k+1})\\
\bmat{x_k-x_\star\\x_{k-1}-x_\star \\ \nabla f(y_k)}^\tp N_k  \bmat{x_k-x_\star\\x_{k-1}-x_\star \\ \nabla f(y_k) } &\le f(x_\star)-f(x_{k+1})
\end{align*}
where $M_k\defeq\tilde{M}_k\otimes I_p$, $N_k\defeq\tilde{N}_k\otimes I_p$, and $\tilde{M}_k, \tilde{N}_k$ are defined by
\begin{align}
\label{eq:NAGsubM}
\tilde{M}_k&\defeq\bmat{0& 0 & -\frac{1}{2}\beta_k\\[1mm] 0  & 0 & \frac{1}{2}\beta_k\\[1mm] -\frac{1}{2}\beta_k & \frac{1}{2}\beta_k & \frac{1}{2L}}, \\
\label{eq:NAGsubN}
\tilde{N}_k&\defeq\bmat{0& 0 & -\frac{1}{2}(1+\beta_k)\\[1mm] 0  & 0 & \frac{1}{2}\beta_k\\[1mm] -\frac{1}{2}(1+\beta_k) & \frac{1}{2}\beta_k & \frac{1}{2L}}.
\end{align}
Given any nondecreasing sequence $\{\mu_k\}$, one can define the supply rate as \eqref{eq:supplysub} with the particular choice $X_k\defeq\mu_k M_k+(\mu_{k+1}-\mu_k) N_k$ for all $k$. Then this supply rate satisfies the condition
\begin{multline}
\label{eq:NAGSupplysub}
S(\xi_k, w_k)\le \mu_k (f(x_k)-f_\star)\\
-\mu_{k+1}(f(x_{k+1})-f_\star).
\end{multline}
\end{lem}
\begin{proof}
The proof is very similar to the proof of Lemma \ref{lem:SupplyNAG} with an extra condition $m=0$. A detailed proof is presented in the supplementary material.
\end{proof}

\begin{thm}
\label{thm:DINAGsub}
Consider the LTV dynamical system \eqref{eq:sys2}.
If there exist matrices $\{P_k\}$ with $P_k \ge 0$ and a nondecreasing sequence of nonnegative scalars $\{\mu_k\}$ such that
\begin{multline}\label{eq:lmiNAGsub}
\bmat{A_k^\tp P_{k+1} A_k- P_k &  A_k^\tp P_{k+1} B_k\\[1mm]
B_k^\tp P_{k+1} A_k & B_k^\tp P_{k+1} B_k}\\
-\mu_k M_k -(\mu_{k+1}-\mu_k) N_k \le 0
\end{multline}
then we have $V_{k+1}(\xi_{k+1})-V_k(\xi_k)\le S_k(\xi_k, w_k)$ with the storage function $V_k(\xi_{k})\defeq\xi_k^\tp P_k \xi_k$ and the supply rate \eqref{eq:supplysub} using $X_k\defeq\mu_k M_k+(\mu_{k+1}-\mu_k) N_k$ for all $k$.
In addition, if this supply rate satisfies \eqref{eq:NAGSupplysub}, we have 
\begin{align}
\label{eq:subNAGfinal}
f(x_k)-f_\star\le \frac{\mu_0(f(x_0)-f_\star)+V_0(\xi_0)}{\mu_k}
\end{align}
\end{thm}
\begin{proof}
Based on the state-space model \eqref{eq:sys2}, we can left and
right multiply \eqref{eq:lmiNAGsub}  by $\bmat{\xi_k^\tp & w_k^\tp}$ and $\bmat{\xi_k^\tp & w_k^\tp}^\tp$, and directly obtain the dissipation inequality. Combining this dissipation inequality with \eqref{eq:NAGSupplysub}, we can show
\[
 V_{k+1}(\xi_{k+1})  + \mu_{k+1}(f_{k+1}-f_\star) \le V_k(\xi_k)+\mu_k (f_{k}-f_\star). 
\]
Summing the above inequality as in the proof of Theorem~\ref{thm6} and using the fact that $P_k \ge 0$ for all $k$ yields the result.
\end{proof}

We are now ready to show the $O(1/k^2)$ rate result for Nesterov's accelerated method. Set $\mu_k\defeq(\zeta_{k-1})^2$ and $P_k\defeq\frac{L}{2}\bmat{\zeta_{k-1} \\ 1-\zeta_{k-1}}\bmat{\zeta_{k-1} & 1-\zeta_{k-1}}$. Note that $P_k \ge 0$ and $\mu_{k+1}-\mu_k=\zeta_k$. It is straightforward to verify that this choice of $\{P_k, \mu_k\}$ makes the left side of \eqref{eq:lmiNAGsub} the zero matrix and hence \eqref{eq:subNAGfinal} holds. Using the fact that $\zeta_{k-1}\ge k/2$ (easily proved by induction), we have $\mu_k \ge k^2/4$ and the $O(1/k^2)$ rate for Nesterov's method follows.

\begin{rem}
\label{rem:unify}
Theorem \ref{thm:DINAGsub} is quite general. The infinite family of LMIs \eqref{eq:lmiNAGsub} can also be applied for linear rate analysis and collapses down to the single LMI \eqref{eq:lmi1} in that case. To apply \eqref{eq:lmiNAGsub} to linear rate analysis, one needs to slightly modify \eqref{eq:NAGsubM}--\eqref{eq:NAGsubN} such that the strong convexity parameter $m$ is incorporated into the formulas of $M_k, N_k$. By setting $\mu_k\defeq\rho^{-2k}$ and $P_k\defeq\rho^{-2k} P$, then the LMI \eqref{eq:lmiNAGsub} is the same for all $k$ and we recover \eqref{eq:lmi1}. This illustrates how the infinite number of LMIs \eqref{eq:lmiNAGsub} can collapse to a single LMI under special circumstances.
\end{rem}

\section{Continuous-time Dissipation Inequality}
Finally, we briefly discuss dissipativity theory for the continuous-time ODEs used in optimization research. Note that dissipativity theory was first introduced in \citet{willems72a, willems72b} in the context of continuous-time systems. We denote continuous-time variables in upper case.
Consider a continuous-time state-space model
\begin{align}
\label{eq:sys3}
\dot{\Lambda}(t)=A(t)\Lambda(t)+B(t)W(t)
\end{align}
where $\Lambda(t)$ is the state, $W(t)$ is the input, and $\dot{\Lambda}(t)$ denotes the time derivative of $\Lambda(t)$.
In continuous-time, the supply rate is a function $S:\R^{n_\Lambda}\times \R^{n_W}\times\R_+\to \R$ that assigns a scalar to each possible state and input pair. Here, we allow $S$ to also depend on time $t \in \R_+$. To simplify our exposition, we will omit the explicit time dependence $(t)$ from our notation.

\begin{defn}
The dynamical system \eqref{eq:sys3} is dissipative with respect to the supply rate $S$ if there exists a function $V:\R^{n_\Lambda}\times \R_+\to\R_+$ such that $V(\Lambda,t) \ge 0$ for all $\Lambda\in\R^{n_\Lambda}$ and $t\ge 0$ and
\begin{align}
\label{eq:DIcon}
\dot{V}(\Lambda,t)\le S(\Lambda, W,t)
\end{align}
for every trajectory of \eqref{eq:sys3}. Here, $\dot V$ denotes the Lie derivative (or total derivative); it accounts for $\Lambda$'s dependence on $t$. The function $V$ is called a storage function, and \eqref{eq:DIcon} is a (continuous-time) dissipation inequality.
\end{defn}

For any given quadratic supply rate,  one can automatically construct the continuous-time dissipation inequality using semidefinite programs. The following result is standard in the controls literature.

\begin{thm}
\label{thm:DIcon}
Suppose $X(t)\in \R^{(n_{\Lambda}+n_W)\times (n_{\Lambda}+n_W)}$ and $X(t)^\tp\! =\! X(t)$ for all $t$. Consider the quadratic supply rate
\begin{align}
\label{eq:supplycon}
S(\Lambda, W,t) \defeq \bmat{\Lambda \\ W}^\tp X \bmat{\Lambda \\ W}
\quad\text{for all }t.
\end{align}
If there exists a family of matrices $P(t)\in \R^{n_{\Lambda}\times n_{\Lambda}}$ with $P(t)\ge 0$ such that
\begin{align}
\label{eq:lmicon}
\bmat{A^\tp P +P A +\dot{P} & P B \\ B^\tp P  & 0 }-X\le 0
\quad\text{for all }t.
\end{align}
Then we have $\dot{V}(\Lambda,t)\le S(\Lambda, W,t)$ with the storage function defined as $V(\Lambda,t)\defeq\Lambda^\tp P \Lambda$.
\end{thm}
\begin{proof}
Based on the state-space model \eqref{eq:sys3}, we can apply the product rule for total derivatives and obtain
\begin{align*}
\dot{V}(\Lambda,t)&=\dot{\Lambda}^\tp P \Lambda+\Lambda^\tp P \dot{\Lambda}+\Lambda^\tp \dot{P} \Lambda\\
&=\bmat{\Lambda \\ W}^\tp \bmat{A^\tp P +P A +\dot{P} & P B \\ B^\tp P  & 0 }\bmat{\Lambda \\ W}.
\end{align*}
Hence we can left and right multiply \eqref{eq:lmicon}  by $\bmat{\Lambda^\tp & W^\tp}$ and $\bmat{\Lambda^\tp & W^\tp}^\tp$ and obtain the desired conclusion.
\end{proof}
 
The algebraic structure of the LMI \eqref{eq:lmicon} is simpler than that of its discrete-time counterpart \eqref{eq:lmi1} because for given $P$, the continuous-time LMI is linear in $A,B$ rather than being quadratic. This may explain why continuous-time ODEs are sometimes more amenable to analytic approaches than their discretized counterparts.

We demonstrate the utility of \eqref{eq:lmicon} on the continuous-time limit of Nesterov's accelerated method in \citet{Su2016}:
\begin{align}\label{ode}
\ddot{Y}+\frac{3}{t} \dot{Y}+\nabla f(Y)=0,
\end{align}
which we rewrite as~\eqref{eq:sys3} with $\Lambda \defeq \bmat{\dot Y^\tp & Y^\tp-x_\star^\tp}^\tp$, $W\defeq\nabla f(Y)$, $x_\star$ is a point satisfying $\nabla f(x_\star)=0$, and $A,B$ are defined by
\[
A(t) \defeq \bmat{-\frac{3}{t}I_p & 0_p \\ I_p & 0_p},\quad
B(t) \defeq \bmat{-I_p\\ 0_p}.
\]
Suppose $f$ is convex and set $f_\star\defeq f(x_\star)$.
\citet[Theorem 3]{Su2016} constructs the Lyapunov function $\mathcal{V}(Y,t)\defeq t^2(f(Y)-f_\star)+2\norm{Y+\frac{t}{2}\dot{Y}-x_\star}^2$ to show that $\dot{\mathcal{V}}\le 0$ and then directly demonstrate a $O(1/t^2)$ rate for the ODE~\eqref{ode}. To illustrate the power of the dissipation inequality, we use the LMI \eqref{eq:lmicon} to recover this Lyapunov function.
Denote $G(Y,t) \defeq t^2(f(Y)-f_\star)$. 
Note that convexity implies $f(Y)-f_\star\le \nabla f(Y)^\tp(Y-x_\star)$, which we rewrite as
\begin{align*}
2t(f(Y)-f_\star)\le \bmat{\dot{Y} \\ Y-x_\star \\ W}^\tp\!\!
\bmat{0_p & 0_p & 0_p \\ 0_p & 0_p & tI_p\\ 0_p & tI_p & 0_p}\!\!
\bmat{\dot{Y} \\ Y-x_\star \\ W}\!.
\end{align*}
Since $\dot{G}(Y,t)=2t(f(Y)-f_\star)+t^2\nabla f(Y)^\tp \dot{Y}$, we have
\begin{align*}
\dot{G}\le \bmat{\dot{Y} \\ Y-x_\star \\ W}^\tp \bmat{0_p & 0_p & \frac{t^2}{2}I_p \\ 0_p & 0_p & tI_p\\ \frac{t^2}{2}I_p & tI_p & 0_p}\bmat{\dot{Y} \\ Y-x_\star \\ W}.
\end{align*}
Now choose the supply rate $S$ as \eqref{eq:supplycon} with $X(t)$ given by
\begin{align*}
X(t) \defeq - \bmat{0_p & 0_p & \frac{t^2}{2}I_p \\ 0_p & 0_p & tI_p\\ \frac{t^2}{2}I_p & tI_p & 0_p}.
\end{align*}
Clearly $S(\Lambda, W,t)\le -\dot{G}(Y,t)$. Now we can  choose $P(t)\defeq 2\bmat{\frac{t}{2}I_p & I_p}^\tp\bmat{\frac{t}{2}I_p & I_p}$.
Substituting $P$ and $X$ into \eqref{eq:lmicon}, 
the left side of \eqref{eq:lmicon} becomes identically zero. Therefore, $\dot{V}(\Lambda,t)\le S(\Lambda, W,t)\le -\dot{G}(Y,t)$ with the storage function $V(\Lambda,t)\defeq\Lambda^\tp P \Lambda$.
By defining the Lyapunov function $\mathcal{V}(Y,t)\defeq V(Y,t)+G(Y,t)$, we immediately obtain $\dot{\mathcal{V}}\le 0$ and also recover the same Lyapunov function used in \citet{Su2016}.

\begin{rem}
As in the discrete-time case, the infinite family of LMIs~\eqref{eq:lmicon} can also be reduced to a single LMI for the linear rate analysis of continuous-time ODEs. For further discussion on the topic of continuous-time exponential dissipation inequalities, see~\citet{Hu16}.
\end{rem}

\section{Conclusion and Future Work}
In this paper, we developed new notions of dissipativity theory for understanding of Nesterov's accelerated method. Our approach enjoys advantages of both the IQC framework \citep{Lessard2014} and the discretization approach \citep{wilson2016} in the sense that our proposed LMI condition is simple enough for analytical rate analysis of Nesterov's method and can also  be easily generalized to more complicated algorithms. Our approach also gives an intuitive interpretation of the convergence behavior of Nesterov's  method using an energy dissipation perspective.

One potential application of our dissipativity theory is for the design of accelerated methods that are robust to gradient noise. This is similar to the algorithm design work in \citet[\S 6]{Lessard2014}. However, compared with the IQC approach in \citet{Lessard2014}, our dissipativity theory leads to smaller LMIs. This can be beneficial since smaller LMIs are generally easier to solve analytically. In addition, the IQC approach in \citet{Lessard2014} is only applicable to strongly-convex objective functions while our dissipativity theory may facilitate the design of robust algorithm for weakly-convex objective functions. The dissipativity framework may also lead to the design of adaptive or time-varying algorithms.

\section*{Acknowledgements}

Both authors would like to thank the anonymous reviewers for helpful suggestions that improved the clarity and quality of the final manuscript.

This material is based upon work supported by the National Science Foundation under Grant No. 1656951. Both authors also acknowledge support from the Wisconsin Institute for Discovery, the College of Engineering, and the Department of Electrical and Computer Engineering at the University of Wisconsin--Madison.


\bibliography{IQCandSOS}

\begin{thebibliography}{17}
\providecommand{\natexlab}[1]{#1}
\providecommand{\url}[1]{\texttt{#1}}
\expandafter\ifx\csname urlstyle\endcsname\relax
  \providecommand{\doi}[1]{doi: #1}\else
  \providecommand{\doi}{doi: \begingroup \urlstyle{rm}\Url}\fi

\bibitem[Bubeck(2015)]{bubeck2015}
Bubeck, S.
\newblock Convex optimization: Algorithms and complexity.
\newblock \emph{Foundations and Trends{\textregistered} in Machine Learning},
  8\penalty0 (3-4):\penalty0 231--357, 2015.

\bibitem[Bubeck et~al.(2015)Bubeck, Lee, and Singh]{bubeck2015geometric}
Bubeck, S., Lee, Y., and Singh, M.
\newblock A geometric alternative to {N}esterov's accelerated gradient descent.
\newblock \emph{arXiv preprint arXiv:1506.08187}, 2015.

\bibitem[Drusvyatskiy et~al.(2016)Drusvyatskiy, Fazel, and
  Roy]{drusvyatskiy2016}
Drusvyatskiy, D., Fazel, M., and Roy, S.
\newblock An optimal first order method based on optimal quadratic averaging.
\newblock \emph{arXiv preprint arXiv:1604.06543}, 2016.

\bibitem[Fazlyab et~al.(2017)Fazlyab, Ribeiro, Morari, and Preciado]{morari}
Fazlyab, M., Ribeiro, A., Morari, M., and Preciado, V.~M.
\newblock Analysis of optimization algorithms via integral quadratic
  constraints: Nonstrongly convex problems.
\newblock \emph{arXiv preprint arXiv:1705.03615}, 2017.

\bibitem[Flammarion \& Bach(2015)Flammarion and Bach]{flammarion2015}
Flammarion, N. and Bach, F.
\newblock From averaging to acceleration, there is only a step-size.
\newblock In \emph{COLT}, pp.\  658--695, 2015.

\bibitem[Hu \& Seiler(2016)Hu and Seiler]{Hu16}
Hu, B. and Seiler, P.
\newblock Exponential decay rate conditions for uncertain linear systems using
  integral quadratic constraints.
\newblock \emph{IEEE Transactions on Automatic Control}, 61\penalty0
  (11):\penalty0 3561--3567, 2016.

\bibitem[Lessard et~al.(2016)Lessard, Recht, and Packard]{Lessard2014}
Lessard, L., Recht, B., and Packard, A.
\newblock Analysis and design of optimization algorithms via integral quadratic
  constraints.
\newblock \emph{SIAM Journal on Optimization}, 26\penalty0 (1):\penalty0
  57--95, 2016.

\bibitem[Megretski \& Rantzer(1997)Megretski and Rantzer]{Megretski1997}
Megretski, A. and Rantzer, A.
\newblock System analysis via integral quadratic constraints.
\newblock \emph{IEEE Transactions on Automatic Control}, 42:\penalty0 819--830,
  1997.

\bibitem[Megretski et~al.(2010)Megretski, J\"{o}nsson, Kao, and
  Rantzer]{megretskiCSH}
Megretski, A., J\"{o}nsson, U., Kao, C.~Y., and Rantzer, A.
\newblock \emph{Control Systems Handbook}, chapter 41: Integral Quadratic
  Constraints.
\newblock CRC Press, 2010.

\bibitem[Nesterov(2003)]{YEN03a}
Nesterov, Y.
\newblock \emph{Introductory Lectures on Convex Optimization: A Basic Course}.
\newblock Kluwer Academic Publishers, 2003.

\bibitem[Polyak(1987)]{polyak}
Polyak, B.~T.
\newblock Introduction to optimization.
\newblock \emph{Optimization Software}, 1987.

\bibitem[Seiler(2015)]{seiler13}
Seiler, P.
\newblock Stability analysis with dissipation inequalities and integral
  quadratic constraints.
\newblock \emph{IEEE Transactions on Automatic Control}, 60\penalty0
  (6):\penalty0 1704--1709, 2015.

\bibitem[Su et~al.(2016)Su, Boyd, and Cand{{\`e}}s]{Su2016}
Su, W., Boyd, S., and Cand{{\`e}}s, E.
\newblock A differential equation for modeling {N}esterov's accelerated
  gradient method: Theory and insights.
\newblock \emph{Journal of Machine Learning Research}, 17\penalty0
  (153):\penalty0 1--43, 2016.

\bibitem[Wibisono et~al.(2016)Wibisono, Wilson, and Jordan]{wibisono2016}
Wibisono, A., Wilson, A., and Jordan, M.
\newblock A variational perspective on accelerated methods in optimization.
\newblock \emph{Proceedings of the National Academy of Sciences}, pp.\
  201614734, 2016.

\bibitem[Willems(1972{\natexlab{a}})]{willems72a}
Willems, J.C.
\newblock Dissipative dynamical systems {Part I}: General theory.
\newblock \emph{Archive for Rational Mech. and Analysis}, 45\penalty0
  (5):\penalty0 321--351, 1972{\natexlab{a}}.

\bibitem[Willems(1972{\natexlab{b}})]{willems72b}
Willems, J.C.
\newblock Dissipative dynamical systems {Part II}: Linear systems with
  quadratic supply rates.
\newblock \emph{Archive for Rational Mech. and Analysis}, 45\penalty0
  (5):\penalty0 352--393, 1972{\natexlab{b}}.

\bibitem[Wilson et~al.(2016)Wilson, Recht, and Jordan]{wilson2016}
Wilson, A., Recht, B., and Jordan, M.
\newblock A {L}yapunov analysis of momentum methods in optimization.
\newblock \emph{arXiv preprint arXiv:1611.02635}, 2016.

\end{thebibliography}
\bibliographystyle{icml2017}
\clearpage
\setcounter{section}{0}
\renewcommand{\thesection}{\Alph{section}}

\onecolumn
\begin{center}
{\Large\bf Supplementary Material}
\end{center}

We will make use of the following result throughout this section.
\begin{lemm}\label{lem:S1}
Suppose $f$ is $L$-smooth and $m$-strongly convex. Then for all $x,y$ the following inequalities hold.
\begin{subequations}\label{ee}
\begin{align}
\tag{S1}\label{e1a}
f(x)-f(y) &\ge \nabla f(y)^\tp (x-y) + \frac{m}{2}\norm{x-y}^2 \\
\tag{S2}\label{e1b}
f(y)-f(x) &\ge \nabla f(y)^\tp (y-x) - \frac{L}{2}\norm{y-x}^2
\end{align}
\end{subequations}
\end{lemm}
\begin{proof}
These inequalities follow from the definitions of $L$-smoothness and $m$-strong convexity.
\end{proof}

\section{Proof of Lemma 3}

Applying \eqref{e1a} with $(x,y) \mapsto (x_k,y_k)$, we obtain
\begin{align*}
f(x_k)-f(y_k) \ge \nabla f(y_k)^\tp (x_k-y_k) + \frac{m}{2}\norm{x_k-y_k}^2.
\end{align*}
Applying \eqref{e1b} with $(x,y) \mapsto (y_k-\alpha\nabla f(y_k),y_k)$, we obtain
\begin{align*}
f(y_k)-f(y_k-\alpha\nabla f(y_k)) \ge \frac{\alpha}{2}(2-L\alpha) \norm{\nabla f(y_k)}^2.
\end{align*}
Summing these inequalities, we obtain:
\begin{align*}\label{e3}\tag{S3}
f(x_k)-f(y_k-\alpha\nabla f(y_k)) \ge \nabla f(y_k)^\tp (x_k-y_k)+\frac{m}{2}\norm{x_k-y_k}^2 +\frac{\alpha}{2}(2-L\alpha) \norm{\nabla f(y_k)}^2.
\end{align*}
Substituting $x_{k+1}=y_k-\alpha\nabla f(y_k)$ in the left-hand side of~\eqref{e3}, we can rewrite it as
\begin{equation}\label{e4}\tag{S4}
\frac{1}{2} \bmat{x_k-y_k \\ \nabla f(y_k)}^\tp \left(\bmat{ m  & 1 \\  1 & \alpha(2-L\alpha) } \otimes I_p\right) \bmat{x_k-y_k \\ \nabla f(y_k) } \le f(x_{k})-f(x_{k+1}).
\end{equation}
Substituting $y_k=(1+\beta)x_k-\beta x_{k-1}$ into~\eqref{e4}, we obtain
\begin{equation*}
\frac{1}{2}\bmat{x_k-x_\star\\x_{k-1}-x_\star \\ \nabla f(y_k)}^\tp\left(\bmat{ \beta^2 m & -\beta^2 m & -\beta\\ -\beta^2 m  & \beta^2 m & \beta\\ -\beta & \beta & \alpha(2-L\alpha)} \otimes I_p \right)\bmat{x_k-x_\star\\x_{k-1}-x_\star \\ \nabla f(y_k) } \le f(x_{k})-f(x_{k+1}),
\end{equation*}
which directly leads to the formulation of $\tilde{X}_1$ in Lemma 3.
Similarly, we apply \eqref{e1a} with $(x,y)\mapsto (x_\star, y_k)$ and obtain
\begin{equation*}
\frac{1}{2}\bmat{x_k-x_\star\\x_{k-1}-x_\star \\ \nabla f(y_k)}^\tp\left(\bmat{ (1+\beta)^2 m & -\beta(1+\beta) m & -(1+\beta)\\ -\beta(1+\beta) m  & \beta^2 m & \beta\\ -(1+\beta) & \beta & \alpha(2-L\alpha)}\otimes I_p\right)  \bmat{x_k-x_\star\\x_{k-1}-x_\star \\ \nabla f(y_k) } \le f(x_\star)-f(x_{k+1})
\end{equation*}
which directly leads to the formulation of $\tilde{X}_2$ in Lemma 3. The rest of the proof is straightforward. Actually,
we can choose $\tilde{X}\defeq\rho^2\tilde{X}_1+(1-\rho^2)\tilde{X}_2$ and we directly obtain
\begin{equation*}
\bmat{x_k-x_\star\\x_{k-1}-x_\star \\ \nabla f(y_k)}^\tp \left(\tilde{X}\otimes I_p\right) \bmat{x_k-x_\star\\x_{k-1}-x_\star \\ \nabla f(y_k) } \le -(f(x_{k+1})-f(x_\star))+\rho^2(f(x_k)-f(x_\star)).
\end{equation*}

Specifically, $\tilde{X}$ may be computed as
\begin{equation*}
\tilde{X}=\frac{1}{2}\bmat{ (1+\beta)^2 m-(1+2\beta)m\rho^2 & (\rho^2-1-\beta)\beta m & \rho^2-1-\beta\\ (\rho^2-1-\beta)\beta m & \beta^2 m& \beta\\ \rho^2-1-\beta & \beta & \alpha(2-L\alpha)} .
\end{equation*}\qedhere

\section{Proof of Lemma 5}

Applying~\eqref{e1b} with $(x,y)\mapsto (x_{k+1},y_k)$, and making the substitutions $x_{k+1}=(1+\beta)x_k-\beta x_{k-1}-\alpha\nabla f(y_k)$ and $y_k=(1+\eta)x_k-\eta x_{k-1}$, we obtain:
\begin{align}
&\hspace{-5mm}f(y_k) - f(x_{k+1})\ge \nabla f(y_k)^\tp(y_k-x_{k+1})-\frac{L}{2}\norm{x_{k+1}-y_k}^2 \notag\\
&=\nabla f(y_k)^\tp((\beta-\eta)(x_{k-1}-x_k)+\alpha \nabla f(y_k))-\frac{L}{2}\norm{(\beta-\eta)(x_{k-1}-x_k)+\alpha \nabla f(y_k)}^2 \notag\\
&=\frac{1}{2}\bmat{x_k-x_\star\\x_{k-1}-x_\star \\ \nabla f(y_k)}^\tp\left(\bmat{ -L(\beta-\eta)^2 & L(\beta-\eta)^2  & -(1-L\alpha)(\beta-\eta)\\ L(\beta-\eta)^2  & -L(\beta-\eta)^2 & (1-L\alpha)(\beta-\eta)\\ -(1-L\alpha)(\beta-\eta) & (1-L\alpha)(\beta-\eta) & \alpha(2-L\alpha)} \otimes I_p \right)\bmat{x_k-x_\star\\x_{k-1}-x_\star \\ \nabla f(y_k) }
\label{S5}\tag{S5}
\end{align}
Applying~\eqref{e1a} with $(x,y)\mapsto (x_k,y_k)$ and substituting   $y_k=(1+\eta)x_k-\eta x_{k-1}$, we obtain:
\begin{align}
f(x_k) - f(y_k) &\ge 
\nabla f(y_k)^\tp (x_k-y_k)+\frac{m}{2}\norm{x_k-y_k}^2 \notag \\
&=\eta \nabla f(y_k)^\tp(x_{k-1}-x_k)+\frac{m\eta^2}{2}\norm{x_{k-1}-x_k}^2 \notag\\
&=\frac{1}{2}\bmat{x_k-x_\star\\x_{k-1}-x_\star \\ \nabla f(y_k)}^\tp\left(\bmat{ \eta^2m & -\eta^2m  & -\eta\\ -\eta^2m & \eta^2m & \eta\\ -\eta & \eta & 0} \otimes I_p \right)\bmat{x_k-x_\star\\x_{k-1}-x_\star \\ \nabla f(y_k) }
\label{S6}\tag{S6}
\end{align}
Applying \eqref{e1a} with $(x,y)\mapsto(x_\star,y_k)$ and again  substituting   $y_k=(1+\eta)x_k-\eta x_{k-1}$, we obtain:
\begin{align}
f(x_\star)-f(y_k) &\ge \nabla f(y_k)^\tp (x_\star-y_k)+\frac{m}{2}\norm{x_\star-y_k}^2 \notag\\
&= -\nabla f(y_k)^\tp((1+\eta)(x_k-x_\star)-\eta (x_{k-1}-x_\star))+\frac{m}{2}\norm{(1+\eta)(x_k-x_\star)-\eta(x_{k-1}-x_\star)}^2\notag\\
&= \frac{1}{2}\bmat{x_k-x_\star\\x_{k-1}-x_\star \\ \nabla f(y_k)}^\tp
\left(\bmat{ (1+\eta)^2m & -\eta(1+\eta)m  & -(1+\eta)\\
 -\eta(1+\eta)m & \eta^2m & \eta\\
 -(1+\eta) & \eta & 0} \otimes I_p \right)
 \bmat{x_k-x_\star\\x_{k-1}-x_\star \\ \nabla f(y_k) }
 \label{S7}\tag{S7}
\end{align}
By adding \eqref{S5}--\eqref{S7} with the definitions of $\tilde X_1$, $\tilde X_2$, and $\tilde X_3$ in Lemma~5, we obtain:
\begin{align*}
\bmat{x_k-x_\star\\x_{k-1}-x_\star \\ \nabla f(y_k)}^\tp\left((\tilde{X}_1+\tilde{X}_2) \otimes I_p \right)\bmat{x_k-x_\star\\x_{k-1}-x_\star \\ \nabla f(y_k) } &\le f(x_k)-f(x_{k+1})\\
\bmat{x_k-x_\star\\x_{k-1}-x_\star \\ \nabla f(y_k)}^\tp\left((\tilde{X}_1+\tilde{X}_3) \otimes I_p \right)\bmat{x_k-x_\star\\x_{k-1}-x_\star \\ \nabla f(y_k) } &\le f(x_\star)-f(x_{k+1})
\end{align*}
The rest of the proof follows by substituting above expressions into the weighted sum with $\rho^2$.\qedhere

\section{Proof of Lemma 8}
Since $f$ is $L$-smooth and convex, we can use the same proof technique as in Lemma 3 while setting $m=0$ and $\alpha=\frac{1}{L}$. We can thus obtain the following inequalities that parallel~\eqref{e4}.
\begin{align*}
\frac12 \bmat{y_k-x_k \\ \nabla f(y_k)}^\tp \left( \bmat{ 0 &  1 \\ 1 & -\frac{1}{L} } \otimes I_p\right) \bmat{y_k-x_k \\ \nabla f(y_k) } &\ge f(x_{k+1})-f(x_k)\\
\frac12 \bmat{y_k-x_\star \\ \nabla f(y_k)}^\tp \left( \bmat{ 0 &  1 \\ 1 & -\frac{1}{L} } \otimes I_p\right) \bmat{y_k-x_\star \\ \nabla f(y_k) } &\ge f(x_{k+1})-f(x_\star)
\end{align*}
The conclusion of Lemma 8 follows once we substitute $y_k=(1-\beta_k) x_k+\beta_k x_{k-1}$.\qedhere

\end{document}